\documentclass[10pt]{amsart}

\usepackage{epsf,amssymb,times,overpic,amscd}
\usepackage[usenames]{color}
\usepackage[all]{xy}
\usepackage{hyperref}

\usepackage{amsfonts}
\usepackage{amsmath}
\usepackage{graphicx}
\usepackage{epsfig}
\usepackage{epstopdf}

\newcommand{\ba}{\begin{array}}
\newcommand{\ea}{\end{array}}
\newcommand{\be}{\begin{enumerate}}
\newcommand{\ee}{\end{enumerate}}

\newtheorem{thm}{Theorem}[section]
\newtheorem{prop}[thm]{Proposition}
\newtheorem{lemma}[thm]{Lemma}
\newtheorem{cor}[thm]{Corollary}
\newtheorem{conj}[thm]{Conjecture}

\theoremstyle{definition}
\newtheorem{defn}[thm]{Definition}

\theoremstyle{remark}

\newtheorem{rmk}[thm]{Remark}
\newtheorem{example}[thm]{Example}

\numberwithin{equation}{section}
\numberwithin{thm}{section}

\newcommand{\C}{\mathcal{C}}
\newcommand{\D}{\mathcal{D}}

\newcommand{\Z}{\mathbb{Z}}

\newcommand{\K}{\mathbf{k}}
\newcommand{\Ks}{\mathbf{k}^{\times}}

\newcommand{\Hom}{\operatorname{Hom}}
\newcommand{\End}{\operatorname{End}}

\newcommand{\lm}{\lambda}

\newcommand{\n}{\noindent}
\newcommand{\ot}{\otimes}
\newcommand{\bt}{\boxtimes}

\newcommand{\ra}{\rightarrow}
\newcommand{\xra}{\xrightarrow}

\newcommand{\mf}{\mathbf}
\newcommand{\op}{\operatorname}

\newcommand{\lan}{\langle}
\newcommand{\ran}{\rangle}

\newcommand{\ti}{\tilde}
\newcommand{\al}{\alpha}

\newcommand{\om}{\omega}
\newcommand{\tr}{R}
\newcommand{\de}{\delta}
\newcommand{\bo}{\boxplus}

\newcommand{\ta}{\tilde{A}}
\newcommand{\Ra}{\Rightarrow}

\newcommand{\cz}{\mathcal{Z}}
\newcommand{\cv}{\mathcal{V}}
\newcommand{\vc}{\operatorname{1Vec}}
\newcommand{\zv}{\mathcal{Z}(\operatorname{2Vec}_G^{\omega})}

\newcommand{\ve}{\operatorname{2Vec}}
\newcommand{\vgw}{\operatorname{2Vec}_G^{\omega}}

\newcommand{\rep}{\operatorname{2Rep}}

\newcommand{\rh}{\operatorname{2Rep}(C_G(h),\tau_h(\omega))}

\newcommand{\cgh}{C_G(h)}
\newcommand{\gh}{g^{-1}hg}

\newcommand{\cl}{\operatorname{Cl}}

\newcommand{\zb}{\mathcal{Z}(\mathcal{B})}
\newcommand{\cb}{\mathcal{B}}
\newcommand{\bic}{\operatorname{Bicat}}

\begin{document}

\title{The center of monoidal 2-categories in 3+1D Dijkgraaf-Witten Theory}
\author{Liang Kong}
\address{Shenzhen Institute for Quantum Science and Engineering, and Department of Physics, Southern University of Science and Technology, Shenzhen, 518055, China}
\email{kongl@sustech.edu.cn}
\author{Yin Tian}
\address{Yau Mathematical Sciences Center, Tsinghua University, Beijing 100084, China}
\email{yintian@mail.tsinghua.edu.cn}
\author{Shan Zhou}
\address{Department of Physics, University of California, Santa Barbara, CA 93106, USA}
\email{shan\_zhou@physics.ucsb.edu}
\maketitle

\begin{abstract}
In this work, for a finite group $G$ and a 4-cocycle $\omega \in Z^4(G,\K^\times)$, we compute explicitly the center of the monoidal 2-category $\vgw$ of $\omega$-twisted $G$-graded 1-categories of finite dimensional $\K$-vector spaces.
This center gives a precise mathematical description of the topological defects in the associated 3+1D Dijkgraaf-Witten TQFT.
We prove that this center is a braided monoidal 2-category with a trivial sylleptic center.
\end{abstract}




\section{Introduction}

The notion of the center of a monoidal 2-category was introduced long time ago \cite{BN,Cr,KV}.
As far as we know, there is, however, no explicit computation of the centers of any non-trivial monoidal 2-categories. In recent years, the demand for such computation from physics becomes paramount. In this work, we consider a very simple case motivated from the physics of 3+1D topological orders.

Let $\cv$ be the 1-category of finite dimensional $\K$-vector spaces (i.e. $\mathrm{1Vec}$).
The ground field $\K$ is assumed to be $\mathbb{C}$ throughout the paper.
Let $G$ be a finite group and $\omega \in Z^4(G,\K^\times)$ a 4-cocycle. Let $\vgw$ be the 2-category of $G$-graded 1-categories of finite semisimple $\cv$-module categories, equipped with a $\omega$-twisted monoidal structure, which makes $\vgw$ a non-strict monoidal 2-category.
The goal of this paper is to compute the center of $\vgw$ as a braided monoidal 2-category.
It can be viewed as 3+1D Dijkgraaf-Witten theory for a finite group $G$ \cite{DW}.

We give a definition of the center of monoidal 2-categories in Section \ref{sec:def-center}.
It is a weak version of Crans' definition of the center of semi-strict monoidal 2 categories \cite{Cr}.
Our first main result is that the center of a monoidal 2-category is a braided monoidal 2-category, see Theorem \ref{thm def zb}.
We further compute explicitly the center $\zv$ of the monoidal 2-category $\vgw$ in Section \ref{sec zv}.
Although any monoidal 2-category has a semi-strict model \cite{GPS}, it is more convenient for us to consider non semi-strict monoidal 2-categories $\vgw$ when we compute the center of $\vgw$.


The analogue on the level of 1-categories is known as the twisted Drinfeld double of a finite group $G$. Let $\operatorname{1Vec}_G^{\chi}$ be the $\chi$-twisted monoidal 1-category of $G$-graded finite dimensional $\K$-vector spaces for $\chi \in Z^3(G, \Ks)$.
Let $\cl$ be the set of conjugacy classes of $G$, and $\cgh$ be the centralizer of $h \in G$.
There is a the transgression map $\tau_h: C^{k+1}(G, \Ks) \ra C^k(\cgh, \Ks)$.
Willerton used it to give a geometric description of the twisted Drinfeld double, and showed that there is an equivalence of 1-categories:
$$\cz(\operatorname{1Vec}_G^{\chi}) \simeq \bigoplus_{[h] \in \cl} \op{1Rep}(\cgh, \tau_h(\chi)),$$
where $\op{1Rep}(\cgh, \tau_h(\chi))$ is the 1-category of representations of the central extension of $\cgh$ determined by the 2-cocycle $\tau_h(\chi)$ \cite{DPR, Wi}.
Our second result generalizes this from 1-categories to 2-categories.

\begin{thm} \label{thm main}
There is an equivalence of 2-categories:
$$
\zv \simeq \bo_{[h] \in \cl} \rh,
$$
where $\rh$ is the 2-category of right module categories over the monoidal 1-category $\operatorname{1Vec}_{\cgh}^{\tau_h(\om)}$.
\end{thm}


The braided monoidal structure of $\zv$ will be explicitly described in Section \ref{sec:braided zv}.
We expect a similar generalization to $n$-categories.

\begin{conj}
For $\omega\in Z^{n+2}(G,\K^\times)$ and a properly defined notion of an $n$-category, we have an equivalence of $n$-categories:
$$
\mathcal{Z}(n\mathrm{Vec}_G^\omega) \simeq \bo_{[h] \in \cl} \,\, n\mathrm{Rep}(C_G(h), \tau_h(\omega)).
$$
\end{conj}

While we are preparing this paper, a beautiful work on the definition of a {\em fusion 2-category} by Douglas and Reutter appeared online \cite{DR}.
They introduced the notion of the 2-categorical {\em idempotent completion}, which is used to define that of 2-categorical {\em semisimpleness}.
Our result further confirms their definition.
In particular, $\zv$ is idempotent complete and semisimple.
We expect that it is a fusion 2-category.


We next discuss the sylleptic center of braided monoidal 2-categories which is a generalization of Crans' definition in the semi-strict case \cite{Cr} in Section \ref{sec muger}.
Our third result is that the sylleptic center of $\zv$ is trivial. Thus $\zv$ should be an example of the yet-to-be-defined {\em modular tensor 2-category}.

\begin{thm} \label{thm main2}
The sylleptic center of $\zv$ is equivalent to $\ve$ as 2-categories.
\end{thm}

Our motivations of this work are threefold.

(1) It was proposed in \cite{LKW} that $\vgw$ describes an anomalous 2+1D topological order, and its gravitational anomaly (or its bulk) is a 3+1D topological order consisting of topological excitations precisely described by the braided monoidal 2-category $\zv$. The objects in $\zv$ represent string-like topological excitations, 1-morphisms represent particle-like topological excitations and 2-morphisms represent instantons. We compute $\zv$ explicitly and summarize the result in Theorem\,\ref{thm main}. It is also known that the low energy effective theory of this 3+1D topological order is the well-known 3+1D Dijkgraaf-Witten TQFT associated to $(G,\omega)$ \cite{DW}. Therefore, Theorem\,\ref{thm main} also classifies all topological defects in the 3+1D Dijkgraaf-Witten TQFT. In particular, the monoidal 1-category of endomorphisms of the vacuum (i.e. particle-like excitations) is equivalent to the category $\op{Rep}(G)$ of the representations of $G$. Moreover, Theorem\,\ref{thm main} provides an efficient way to detect the physical difference between two anomalous 2+1D topological orders $2\mathrm{Vec}_G^{\omega_1}$ and $2\mathrm{Vec}_G^{\omega_2}$ by measuring their gravitational anomalies (i.e. centers), which are different not only in their braiding structures but also on the level of 2-categories (see Example\,\ref{ex Z2Z2}).

(2) It is well-known that the topological excitations in a 2+1D topological order form a modular tensor 1-category (MTC). The 3+1D analogue of MTC, i.e. the yet-to-be-defined modular tensor 2-category, should include $\zv$ as an example. Our second motivation is to find the correct definition of a (braided) fusion 2-category and that of a modular tensor 2-category.  It is worthwhile to point out that, even in this simple case, in order to reveal the intertwined relation between the braidings and the 4-cocycle $\omega$, it is more convenient to work in the non semi-strict setting.


(3) Our third motivation is to find a categorification of conformal blocks by integrating a modular tensor 2-category over 2-dimensional manifolds via factorization homology (see a recent review \cite{AF} and references therein).
Douglas and Reutter constructed a state-sum invariant for 4-manifolds associated to any fusion 2-category.
We expect that the integration of $\zv$ is related to their invariant associated to $\vgw$.

This work is the first in a series of works on (braided) fusion 2-categories. Our long term goal is to develop a mathematical theory of modular tensor 2-categories and a physical theory of the condensations of topological excitations in 3+1D topological orders. For example, the forgetful functor $\zv \to \vgw$ is precisely the mathematical description of a physical condensation process. 

\medskip
\noindent {\bf Acknowledgement}: Thank the referee for many useful suggestions. Thank Christopher L. Douglas for correcting typos. Liang Kong is partially supported by Guangdong Innovative and Entrepreneurial Research Team Program (Grant No. 2016ZT06D348), by the Science, Technology and Innovation Commission of Shenzhen Municipality (Grant Nos. ZDSYS20170303165926217 and JCYJ20170412152620376), and by the NSFC grant No. 11971219. Yin Tian is partially supported by the NSFC grants No. 11601256 and 11971256. Shan Zhou is partially supported by Simons Foundation grant No. 488629.

\section{The center of monoidal 2-categories} \label{sec:def-center}
In this section, we give a definition of the center of monoidal 2-categories.
It is a weak version of Crans' definition of the center of semi-strict monoidal 2-categories in \cite{Cr}.
We use Gurski's definition of monoidal bicategories and braided monoidal bicategories \cite[Section 2.4]{Gur}.
We prove that the center of a monoidal 2-category is a braided monoidal 2-category in Theorem \ref{thm def zb}.

\medskip
\noindent {\bf Convention}: In this paper, a (braided) monoidal 2-category is defined as a (braided) monoidal bicategory in the sense of Gurski, such that its underlying bicategory is a 2-category, see \cite[Definitions 2.6]{GPS}.

\medskip
We briefly recall the notion of a monoidal bicategory which is defined as a tricategory with one object.
We refer the reader to \cite{Gur} for more detail on tricategories and the coherence.
For two bicategories $\cb, \cb'$, let $\bic(\cb, \cb')$ denote the bicategory of functors $\cb \ra \cb'$, natural transformations and modifications.

Let $\cb=(\cb, \ot, I, \bf{a}, \bf{l}, \bf{r}, \pi, \mu, \rho, \lm)$  be a monoidal 2-category.
It consists of the following data:
\be
\item $\cb$ is a 2-category, $\ot$ is the monoidal bifunctor in $\bic(\cb \times \cb,\cb)$, and $I$ is the tensor unit;
\item $\bf{a}$ is the adjoint equivalence in $\bic(\cb \times \cb \times \cb, \cb)$, consisting of a pair $a: (-\ot-)\ot- \ra -\ot(-\ot-)$ and its adjoint equivalence $a^*: -\ot(-\ot-) \ra (-\ot-)\ot-$;
\item $\bf{l}$ and $\bf{r}$ are the adjoint equivalences in $\bic(\cb, \cb)$, where $l: I \ot - \ra -$ and $r: - \ot I \ra -$;
\item $\pi$ is the invertible modification for $\bf{a}$, and $\mu,\rho,\lm$ are the invertible modifications for $\bf{a}, \bf{l}, \bf{r}$.
\ee
It satisfies certain axioms which are omitted here.
We define the center $\zb$ in three steps: (1) the 2-category; (2) the monoidal structure; and (3) the braiding.

\n {\bf Step 1: the 2-category $\zb$.}

{\bf Objects.} An object $\ta=(A, R_{A,-}, \tr_{(A|-,?)})$ consists of an object $A$ of $\cb$, an adjoint equivalence $R_{A,-}: A \ot - \ra - \ot A$ in $\bic(\cb, \cb)$, and an invertible modification $\tr_{(A|-,?)}$:
\begin{equation*}
\resizebox{0.4\displaywidth}{!}{
\xymatrix{
& (XA)Y \ar[r]^{a} \ar@{}[ddr]|{{\displaystyle \Downarrow} \tr_{(A|X,Y)}} & X(AY) \ar[dr]^{R_{A,Y}}  & \\
(AX)Y \ar[ur]^{R_{A,X}} \ar[dr]^{a} & & & X(YA) \\
& A(XY) \ar[r]^{R_{A, XY}} & (XY)A \ar[ur]^{a} &
}}\end{equation*}
such that the following axiom holds:
\begin{equation} \label{def AXYZ}
\resizebox{0.8\displaywidth}{!}{
\xymatrix{
 ((XA)Y)Z \ar[r]^{a} \ar[ddr]^{a} \ar@{}[dddr]|{{\displaystyle \cong}} & (X(AY))Z \ar[r]^{R_{A,Y}} \ar[dr]_{a} \ar@{}[dd]|{{\displaystyle \Downarrow} \pi} & (X(YA))Z \ar[r]^{a} \ar@{}[d]|{{\displaystyle \cong}} & X((YA)Z) \ar@{}[dd]|{{\displaystyle \Downarrow} \tr_{(A|Y,Z)}} \ar[r]^{a} & X(Y(AZ)) \ar[dd]^{R_{A,Z}} \\
  & & X((AY)Z) \ar[ur]_{R_{A,Y}} \ar[dr]^{a} & & \\
((AX)Y)Z \ar[uu]^{R_{A,X}} \ar[d]^{a} & (XA)(YZ) \ar[rr]^{a} \ar@{}[drr]|{{\displaystyle \Downarrow} \tr_{(A|X,YZ)}} & & X(A(YZ)) \ar[dr]^{R_{A,YZ}}   & X(Y(ZA)) \\
 (AX)(YZ) \ar[r]^{a} \ar[ur]^{R_{A,X}} & A(X(YZ)) \ar[rr]_{R_{A,X(YZ)}} &  & (X(YZ))A \ar[r]^{a} &  X((YZ)A) \ar[u]^{a} \\
 && {\displaystyle \parallel} && \\
 ((XA)Y)Z \ar[r]^{a} \ar@{}[dr]|{{\displaystyle \Downarrow} \tr_{(A|X,Y)}}  & (X(AY))Z \ar[r]^{R_{A,Y}}   & (X(YA))Z \ar[r]^{a} \ar@{}[dr]|{{\displaystyle \Downarrow} \pi} & X((YA)Z) \ar[r]^{a} & X(Y(AZ)) \ar[d]^{R_{A,Z}} \\
((AX)Y)Z \ar[u]^{R_{A,X}}   \ar[d]^{a}  \ar@/_3pc/[dd]_{a}  & ((XY)A)Z \ar[ur]^{a} \ar[rr]^{a} \ar@{}[drr]|{{\displaystyle \Downarrow} \tr_{(A|XY,Z)}} & & (XY)(AZ) \ar[ur]^{a} \ar[dr]^{R_{A,Z}} \ar@{}[r]|{{\displaystyle \cong}} &  X(Y(ZA)) \\
(A(XY))Z \ar[ur]^{R_{A,XY}} \ar[r]^{a} \ar@{}[dr]|{{\displaystyle \Downarrow} \pi} & A((XY)Z) \ar[d]^{a} \ar[rr]_{R_{A,(XY)Z}} \ar@{}[drr]|{{\displaystyle \cong}} & & ((XY)Z)A \ar[d]^{a} \ar[r]^{a} \ar@{}[dr]|{{\displaystyle \Uparrow} \pi} &  (XY)(ZA) \ar[u]^{a} \\
 (AX)(YZ) \ar[r]^{a} & A(X(YZ)) \ar[rr]_{R_{A,X(YZ)}} & & (X(YZ))A \ar[r]^{a} &  X((YZ)A) \ar@/_3pc/[uu]_{a},
}}\end{equation}
where the four isomorphisms ``$\cong$'' are those defining the naturality of $a$ in $\cb$.

{\bf 1-morphisms.} A 1-morphism $(f, R_{f,-}): (A, R_{A,-}, \tr_{(A|-,?)}) \ra (A', R_{A',-}, \tr_{(A'|-,?)})$ consists of a 1-morphism $f: A \ra A'$ in $\cb$, and an invertible modification $R_{f,-}$:
$$\xymatrix{
A'X \ar[r]^{R_{A',X}} & XA'\\
AX \ar[u]^{f} \ar[r]_{R_{A,X}}  \ar@{}[ur]|{\Rightarrow R_{f,X}} &   XA \ar[u]_{f}
}$$
such that the following diagram commutes:
\begin{equation} \label{def 1-mor}
\xymatrix{
(XA')Y \ar[rrrrr]^{a}    & &&& & X(A'Y) \ar[dl]^{R_{A',Y}}  \\
& (A'X)Y \ar[ul]^{R_{A',X}} \ar[r]^{a}  & A'(XY) \ar@{}[ur]|{\Downarrow \tr_{(A'|X,Y)}} \ar[r]^{R_{A',XY}}  & (XY)A' \ar[r]^{a}  & X(YA') & \\
(XA)Y \ar@{}[ur]|{{\displaystyle \Leftarrow} R_{f,X}} \ar@{-->}[rrrrr]  \ar[uu] & &&& & X(AY) \ar[dl]^{R_{A,Y}} \ar[uu] \ar@{}[ul]|{{\displaystyle \Leftarrow} R_{f,Y}} \\
& (AX)Y  \ar[ul]^{R_{A,X}} \ar[r]^{a} \ar[uu] \ar@{}[uuur]|{\cong} & A(XY) \ar[r]^{R_{A,XY}} \ar[uu] \ar@{}[uuur]|{{\displaystyle \Rightarrow} R_{f,XY}} \ar@{}[ur]|{\Downarrow \tr_{(A|X,Y)}} & (XY)A \ar[r]^{a} \ar[uu] \ar@{}[uuur]|{\cong} & X(YA)  \ar[uu] & \\
}
\end{equation}
where all vertical arrows are 1-morphisms induced by $f: A \ra A'$ in $\cb$.

{\bf 2-morphisms.} A 2-morphism $\alpha: (f, R_{f,-}) \Ra (f', R_{f',-})$ is a 2-morphism $\al: f \Ra f'$ in $\cb$
such that $\al \cdot R_{f,-}=R_{f',-} \cdot \al$, i.e. the following diagram commutes:
\begin{equation} \label{def 2-mor}
\xymatrix{
A'X \ar[rrr]^{R_{A',X}} & & & XA' \\
AX \ar[rrr]_{R_{A,X}} \ar@/^1.5pc/[u]^{f}  \ar@{-->}@/_1.5pc/[u]_{f'} \ar@{}[u]|{\Rightarrow \al} & & & XA \ar@/^1.5pc/[u]^{f}  \ar@/_1.5pc/[u]_{f'} \ar@{}[u]|{\Rightarrow \al}
}
\end{equation}
where the 2-isomorphisms in the front and back are $R_{f,X}$ and $R_{f',X}$, respectively.


{\bf Composition of 1-morphisms.} Given two 1-morphisms $(f, R_{f,-})$ and $(g, R_{g,-})$, the composition
$$(g, R_{g,-}) \circ (f, R_{f,-})=(gf, R_{gf,-}),$$
where $gf$ is the composition in $\cb$, and $R_{gf,-}$ is given by the following composition of 2-morphisms:
$$\xymatrix{
A''X \ar[r]^{R_{A'',X}}  & XA'' \\
A'X \ar[r]_{R_{A',X}} \ar[u]^{g} \ar@{}[ur]|{\Rightarrow R_{g,X}} & XA' \ar[u]_{g} \\
AX \ar[u]^{f} \ar[r]_{R_{A,X}}  \ar@{}[ur]|{\Rightarrow R_{f,X}} &   XA. \ar[u]_{f}
}$$

\begin{rmk} \label{rmk unit}
The main difference between our definition and Crans' definition is that we are working with monoidal 2-categories instead of semi-strict monoidal 2-categories.
The monoidal structure of $\vgw$ is not semi-strict when $\om$ is nontrivial.
Although any monoidal 2-category has a semi-strict model \cite{GPS}, we do not know how to compute the center of the semi-strict model of $\vgw$ directly.
On the other hand, by working with non semi-strict associators, we can see explicitly the relation between the braidings and the associators (see Diagram\,(\ref{def AXYZ}) and Eq.\,(\ref{eq new key 2iso})). This relation affects not only the braiding structure but also the objects of the center.
More precisely, the underlying 2-categories of the center of $\vgw$ could be inequivalent for different classes $\om$, see Example \ref{ex Z2Z2}.
\end{rmk}

\vspace{.2cm}
\n {\bf Step 2: the monoidal structure.}
We construct a monoidal 2-category $(\zb, \ot, \tilde{I}, \tilde{\mf{a}}, \tilde{\mf{l}}, \tilde{\mf{r}}, \tilde{\pi}, \tilde{\mu}, \tilde{\lm}, \tilde{\rho})$.

Tensor product of two objects $(A, R_{A,-}, \tr_{(A|-,?)}) \ot (B, R_{B,-}, \tr_{(B|-,?)})=(AB, R_{AB,-}, R_{(AB|-,?)}),$ where $R_{AB,-}$ is an adjoint equivalence given by the composition:
$$(AB)- \xra{a} A(B-) \xra{R_{B,-}} A(-B) \xra{a^*} (A-)B \xra{R_{A,-}} (-A)B \xra{a} -(AB),$$
and $R_{(AB|-,?)}$ is an invertible modification:
\begin{equation} \label{def ABXY}
\resizebox{\displaywidth}{!}{
\xymatrix{
 ((AB)X)Y \ar@{}[drr]|{=} \ar[d]^{a} \ar[rr]^{R_{AB,X}} \ar@/_3pc/[dddd]_{a} &  & (X(AB))Y \ar[rrr]^{a} &&& X((AB)Y) \ar@{}[dr]|{=} \ar[dl]^{a} \ar[rr]^{R_{AB,Y}} & & X(Y(AB))  \\
(A(BX))Y \ar[r]^{R_{B,X}} \ar[d]^{a}_{\pi \Leftarrow}  \ar@{}[dr]|{\cong} &  (A(XB))Y \ar[r]^{a^{*}} \ar[d]^{a}  \ar@{}[dr]|{\Rightarrow \pi} & ((AX)B)Y \ar[d]^{a} \ar@{}[dr]|{\cong} \ar[r]^{R_{A,X}} & ((XA)B)Y \ar[ul]^{a} \ar[d]^{a} \ar@{}[r]|{\Downarrow \pi} & X(A(BY))  \ar[r]^{R_{B,Y}} \ar@{}[dr]|{\cong} & X(A(YB))  \ar[r]^{a^{*}} \ar@{}[dr]|{\Leftarrow \pi} & X((AY)B) \ar[r]^{R_{A,Y}} \ar@{}[dr]|{\cong} & X((YA)B) \ar[u]^{a} \\
A((BX)Y) \ar[r]^{R_{B,X}} \ar[d]^{a} & A((XB)Y)  \ar[dr]^{a}   & (AX)(BY) \ar[d]^{a} \ar[dr]^{R_{B,Y}} \ar[r]^{R_{A,X}} & (XA)(BY)  \ar@{=}[r] \ar@{}[dr]|{\cong} & (XA)(BY) \ar[u]^{a} \ar[r]^{R_{B,Y}} & (XA)(YB) \ar[u]^{a} & (X(AY))B \ar[u]^{a} \ar[r]^{R_{A,Y}} \ar[u]^{a} & (X(YA))B \ar[u]^{a}_{\Rightarrow \pi}  \\
A(B(XY)) \ar[drr]_{R_{B,XY}} & \Downarrow \tr_{(B|X,Y)} & A(X(BY)) \ar@{}[r]|{\cong} \ar[dr]^{R_{B,Y}} & (AX)(YB) \ar[d]^{a} \ar@{=}[r] & (AX)(YB) \ar@{}[r]|{\cong} \ar[ur]^{R_{A,X}} & ((XA)Y)B \ar[ur]^{a} \ar[u]^{a} & \Downarrow \tr_{(A|X,Y)} & ((XY)A)B \ar[d]^{a} \ar[u]^{a} \\
(AB)(XY) \ar[u]^{a} \ar@/_3pc/[rrrrrrr]_{R_{AB,XY}} & & A((XY)B) \ar[r]^{a} \ar@/_0.7pc/[rrr]^{a^*} & A(X(YB)) \ar@{}[ur]|{\Uparrow \pi} \ar@{}[dr]|{=}  & ((AX)Y)B \ar[u]^{a} \ar[r]^{a} \ar[ur]^{R_{A,X}}& (A(XY))B \ar[urr]_{R_{A,XY}} & & (XY)(AB). \ar@/_3pc/[uuuu]_{a} \\
&&&&&&&
}}
\end{equation}

Tensor product of an object $\ta=(A, R_{A,-}, \tr_{(A|-,?)})$ and a 1-morphism $(g, R_{g,-}): (B, R_{B,-}, \tr_{(B|-,?)}) \ra (B', R_{B',-}, \tr_{(B'|-,?)})$ is a 1-morphism $(Ag, R_{Ag,-}): \ta \ti{B} \ra \ta \ti{B'}$, where $Ag: AB \ra AB'$ is the 1-morphism in $\cb$, and $R_{Ag,-}$ is an invertible modification defined by the following diagram:
$$\xymatrix{
(AB')X \ar[r]^{a}  & A(B'X) \ar[r]^{R_{B',X}} & A(XB') \ar[r]^{a^*} & (AX)B' \ar[r]^{R_{A,X}} & (XA)B' \ar[r]^{a} & X(AB') \\
(AB)X \ar[r]_{a} \ar[u] \ar@{}[ur]|{\cong} & A(BX) \ar[r]_{R_{B,X}} \ar[u] \ar@{}[ur]|{R_{g,X}} & A(XB) \ar[r]_{a^*} \ar[u] \ar@{}[ur]|{\cong} & (AX)B \ar[r]_{R_{A,X}} \ar[u] \ar@{}[ur]|{\cong} & (XA)B \ar[r]_{a} \ar[u] \ar@{}[ur]|{\cong} & X(AB) \ar[u]
}$$
where all vertical arrows are 1-morphisms induced by $g$.

Tensor product of a 1-morphism $(f, R_{f,-}): \ta \ra \ta'$ and an object $\ti{B}$ is a 1-morphism $(fB, R_{fB,-}): \ta \ti{B} \ra \ta' \ti{B}$, where $fB: AB \ra A'B$ is the 1-morphism in $\cb$, and $R_{fB,-}$ is an invertible modification:
$$\xymatrix{
(A'B)X \ar[r]^{a}  & A'(BX) \ar[r]^{R_{B,X}} & A'(XB) \ar[r]^{a^*} & (A'X)B \ar[r]^{R_{A',X}} & (XA')B \ar[r]^{a} & X(A'B) \\
(AB)X \ar[r]_{a} \ar[u] \ar@{}[ur]|{\cong} & A(BX) \ar[r]_{R_{B,X}} \ar[u] \ar@{}[ur]|{\cong} & A(XB) \ar[r]_{a^*} \ar[u] \ar@{}[ur]|{\cong} & (AX)B \ar[r]_{R_{A,X}} \ar[u]  \ar@{}[ur]|{R_{f,X}} & (XA)B \ar[r]_{a} \ar[u] \ar@{}[ur]|{\cong} & X(AB) \ar[u]
}$$
where all vertical arrows are 1-morphisms induced by $f$.

The unit $\tilde{I}=(I, R_{I,-}, \tr_{(I|-,?)})$, where $R_{I,-}$ is an adjoint equivalence $I- \xra{l} - \xra{r^*} -I$, and $\tr_{(I|-,?)}$ is an invertible modification:
\begin{equation} \label{diag:unit-1}
\xymatrix{
& XY \ar[r]^{r^*} \ar@{=}[ddr] \ar@{}[dd]|{\Downarrow \lm} & (XI)Y \ar[r]^{a} \ar@{}[ddr]|{\Downarrow \mu} & X(IY)  \ar[r]^{l} & XY \ar@{=}[ddl] \ar[dr]^{r^*} \ar@{}[dd]|{\Downarrow \rho} & \\
(IX)Y \ar[ur]^{l} \ar[dr]^{a} & & & & & X(YI) \\
& I(XY) \ar[r]^{l} & XY \ar@{=}[r] & XY \ar[r]^{r^*} & (XY)I \ar[ur]^{a} &
}
\end{equation}

An associator $\ti{a}: (\ti{A} \ti{B}) \ti{C} \ra \ti{A}(\ti{B}\ti{C})$ is a 1-morphism $(a, R_{a,-})$, where $a: (AB)C \ra A(BC)$ is the associator in $\cb$, and $R_{a,-}$ is an invertible modification:
\begin{equation} \label{def ABCX}
\resizebox{0.85\displaywidth}{!}{
\xymatrix{
 (A(BC))X \ar[rrrrr]^{R_{A(BC),X}} \ar[d]^{a} & & \ar@{}[dr]|{=} & & & X(A(BC))  \\
 A((BC)X) \ar[d]^{a} \ar[rrr]^{R_{BC,X}} & \ar@{}[dr]|{=} & & A(X(BC)) \ar[r]^{a^*} & (AX)(BC) \ar[r]^{R_{A,X}} & (XA)(BC) \ar[u]^{a}  \\
 A(B(CX)) \ar[r]^{R_{C,X}} & A(B(XC)) \ar[r]^{a^*} & A((BX)C) \ar[r]^{R_{B,X}} & A((XB)C) \ar[u]^{a} & &  \\
 & & (A(BX))C \ar[r]^{R_{B,X}} \ar[u]^{a} \ar@{}[ur]|{\cong} & (A(XB))C \ar[r]^{a^*} \ar[u]^{a} \ar@{}[uur]|{\Rightarrow \pi} & ((AX)B)C \ar[r]^{R_{A,X}} \ar@{}[uur]|{\cong} \ar[uu]^{a} & ((XA)B)C \ar[d]^{a} \ar[uu]^{a}   \\
  (AB)(CX) \ar@{}[uur]|{\cong} \ar[uu]^{a} \ar[r]^{R_{C,X}} & (AB)(XC) \ar[r]^{a*} \ar[uu]^{a} \ar@{}[uur]|{\Leftarrow \pi} & ((AB)X)C \ar[u]^{a} \ar[rrr]^{R_{AB,X}} \ar@{}[dr]|{=} & \ar@{}[ur]|{=} & &  (X(AB))C \ar[d]^{a} \\
((AB)C)X \ar[u]^{a} \ar@/^3pc/[uuuuu]^{a}_{\Rightarrow \pi} \ar[rrrrr]_{R_{A(BC),X}}  &&&&&   X((AB)C) \ar@/_3pc/[uuuuu]_{a}^{\pi \Leftarrow}
}}
\end{equation}

An equivalence $\ti{l}: \ti{I} \ti{A} \ra \ti{A}$ is a 1-morphism $(l, R_{l,-})$, where $l: IA \ra A$ is the equivalence in $\cb$, and $R_{l,-}$ is an invertible modification:
\begin{equation} \label{def Rl}
\xymatrix{
 &  AX \ar[rrrr]^{R_{A,X}} & & & & XA \\
& I(AX) \ar[u]^{l} \ar[r]^{R_{A,X}} \ar@{}[l]|{\Rightarrow \lm} & I(XA) \ar[rr]^{l} \ar[d]^{a^*} \ar@{}[ur]|{\cong} \ar@{}[dr]|{\Uparrow \lm}  & & XA \ar[d]^{r^*} \ar@{=}[ur] \ar@{}[r]|{\Uparrow \mu}  &  \\
  & (IA)X \ar@/^4pc/[uu]^{l} \ar@{}[ur]|{=} \ar@/_2pc/[rrrr]^{R_{IA,X}} \ar[u]^{a} & (IX)A \ar[urr]^{l} \ar[rr]^{R_{I,X}} & & (XI)A \ar[r]^{a} \ar@{}[ul]|{=} & X(IA) \ar[uu]^{l}
}\end{equation}

An equivalence $\ti{r}: \ti{A}\ti{I}  \ra \ti{A}$ is a 1-morphism $(r, R_{r,-})$, where $r: AI \ra A$ is the equivalence in $\cb$, and $R_{r,-}$ is an invertible modification:
\begin{equation} \label{diag:unit-3}
\xymatrix{
 &  AX\ar@{=}[rr] & &AX \ar[r]^{R_{A,X}} \ar@/_.8pc/[dl]_{r^*} \ar@{}[dd]|{\cong} & XA \\
& A(IX) \ar@{}[ur]|{=} \ar[u]^{l} \ar[r]^{R_{I,X}} \ar@{}[l]|{\Leftarrow \mu} & A(XI)  \ar@{}[ur]|{\Leftarrow \rho} \ar[d]_{a^*}  &   &  \\
  & (AI)X \ar@/^4pc/[uu]^{r} \ar@{}[ur]|{=} \ar@/_2pc/[rrr]^{R_{AI,X}} \ar[u]^{a} & (AX)I \ar[uur]_{r} \ar[r]^{R_{A,X}}  & (XA)I \ar[r]^{a} \ar[uur]^{r} \ar@{}[ur]|{\Rightarrow \rho} & X(AI) \ar[uu]^{r}
}\end{equation}

Invertible modifications $\tilde{\pi}, \tilde{\mu}, \tilde{\lm}, \tilde{\rho}$ are defined in the same way as in $\cb$.
We need to show that they are well-defined 2-morphisms in $\zb$, i.e. they satisfy the axiom in (\ref{def 2-mor}).

We check the case of $\ti{\lm}$ in the following and leave other cases to the reader.
The invertible modification $\ti{\lm}: \ti{l} \Rightarrow \ti{l} \circ \ti{a}$ is defined as $\lm: l \Rightarrow l \circ a$ in $\cb$.
We need to show that the following diagram commutes:
\begin{equation} \label{diag:ABX}
\xymatrix{
(AB)X \ar[rr]^{R_{AB,X}}  & & X(AB) \\
((IA)B)X \ar[rr]_{R_{((IA)B,X}} \ar@/^1pc/[u]^{l}  \ar@{-->}@/_1pc/[u]_{l\circ a} \ar@{}[u]|{\Rightarrow \lm} & & X((IA)B) \ar@/^1pc/[u]^{l}  \ar@/_1pc/[u]_{l\circ a} \ar@{}[u]|{\Rightarrow \lm}
}
\end{equation}
where the 2-isomorphisms in the front and back are $R_{l,X}$ and $R_{l \circ a, X}$, respectively.
We decompose the diagram into pieces:
$$
\xymatrix{
(AB)X \ar[r]^{a} & A(BX) \ar[r]^{R_{B,X}}   & A(XB) \ar[r]^{a^*} & (AX)B \ar[r]^{R_{A,X}} & (XA)B \ar[r]^{a} & X(AB) \\
((IA)B)X \ar[r]^{a} \ar@/^1pc/[u]^{l}  \ar@{-->}@/_1pc/[u]_{l\circ a} \ar@{}[u]|{\Rightarrow \lm} & (IA)(BX) \ar[r]^{R_{B,X}} \ar@/^1pc/[u]^{l}  \ar@{-->}@/_1pc/[u]_{l\circ a}  \ar@{}[u]|{\Rightarrow \lm} & (IA)(XB) \ar[r]^{a^*} \ar@/^1pc/[u]^{l}  \ar@{-->}@/_1pc/[u]_{l\circ a} \ar@{}[u]|{\Rightarrow \lm} & ((IA)X)B \ar[r]^{R_{IA,X}} \ar@/^1pc/[u]^{l}  \ar@{-->}@/_1pc/[u]_{l\circ a} \ar@{}[u]|{\Rightarrow \lm} & (X(IA))B \ar[r]^{a} \ar@/^1pc/[u]^{l}  \ar@{-->}@/_1pc/[u]_{r \circ a^*} \ar@{}[u]|{\Rightarrow \mu} & X((IA)B). \ar@/^1pc/[u]^{l}  \ar@/_1pc/[u]_{l\circ a} \ar@{}[u]|{\Rightarrow \lm}
}
$$
The commutativity of each piece follows from the definition of $R_{l,-}$ in (\ref{def Rl}) and the axioms in $\cb$.

\vspace{.2cm}
\n {\bf Step 3: the braiding.}

The braiding of two objects $\ti{A}=(A, R_{A,-}, \tr_{(A|-,?)})$ and $\ti{B}=(B, R_{B,-}, \tr_{(B|-,?)})$ is a 1-morphism $R_{\ti{A}, \ti{B}}=(R_{A, B}, R_{R_{A, B}, -}): \ti{A} \ti{B} \ra \ti{B} \ti{A}$ in $\zb$, where $R_{A,B}=R_{A,-}(B): AB \ra BA$ is the adjoint equivalence in $\cb$, and $R_{R_{A, B}, -}$ is an invertible modification:
\begin{equation} \label{def ABX}
\xymatrix{
&&&&&& \\
(BA)X \ar@/^3pc/[rrrrrr]_{R_{BA,X}} \ar[r]^{a} & B(AX) \ar[r]^{R_{A,X}} & B(XA) \ar[r]^{a^*} & (BX)A \ar@{}[u]|{=} \ar[r]^{R_{B,X}} & (XB)A \ar[rr]^{a} & & X(BA) \\
(AB)X \ar@/_3pc/[rrrrrr]^{R_{AB,X}} \ar[rr]_{a} \ar[u]^{R_{A,B}} & & A(BX) \ar@/^1pc/[ur] \ar@{}[ull]|{\Rightarrow \tr_{(A|B,X)}} \ar[r]_{R_{B,X}} & A(XB) \ar@{}[d]|{=} \ar@/_1pc/[ur] \ar@{}[u]|{\scriptstyle \Rightarrow R_{A,-}(R_{B,X})} \ar[r]_{a^*} & (AX)B \ar[r]_{R_{A,X}}  & (XA)B \ar[r]_{a} \ar@{}[u]|{\Leftarrow \tr_{(A|X,B)}}& X(AB) \ar[u]_{R_{A,B}} \\
&&&&&&
}
\end{equation}

The braiding of an object $\ti{A}=(A, R_{A,-}, \tr_{(A|-,?)})$ and a 1-morphism $(g, R_{g,-})$ is an invertible modification $R_{A,-}(g)$.
The braiding of a 1-morphism $(f, R_{f,-})$ and an object $\ti{B}=(B, R_{B,-}, \tr_{(B|-,?)})$ is an invertible modification $R_{f,-}(B)$.

Two invertible modifications
\begin{equation} \label{def R(A|B,C)}
\tr_{(\ti{A}|\ti{B}, \ti{C})}=\tr_{(A|-,?)}(B,C)=\tr_{(A|B,C)},
\end{equation}
and $\tr_{(\ti{A},\ti{B} | \ti{C})}$ is given by:
\begin{equation} \label{def R(A,B|C)}
\xymatrix{
& A(CB) \ar[r]^{a^*} \ar@{}[ddr]|{=} & (AC)B \ar[dr]^{R_{A,C}}  & \\
A(BC) \ar[ur]^{R_{B,C}} \ar@/_1pc/[dr]_{a^*} & & & (CA)B \ar@/_1pc/[dl]^{a} \\
& (AB)C \ar[r]^{R_{AB, C}} \ar@/_1pc/[ul]^{a} & C(AB). \ar@/_1pc/[ur]_{a^*} &
}
\end{equation}
So $\tr_{(\ti{A},\ti{B} | \ti{C})}$ only differs from the identity by the two units $id_{A(BC)} \Rightarrow aa^*$ and $id_{(CA)B} \Rightarrow a^*a$.

\begin{thm} \label{thm def zb}
The center $\zb$ defined above is a braided monoidal 2-category.
\end{thm}
\begin{proof}
See \cite[Section 2.4]{Gur} for Gurski's definition of a braided monoidal bicategory.
Step 2 makes $\zb$ a monoidal 2-category.
The adjoint equivalence $R: \ot \Rightarrow \ot \circ \tau$ in $\bic(\zb \times \zb, \zb)$, and the invertible modifications $\tr_{(\ti{A}|\ti{B}, \ti{C})}, \tr_{(\ti{A},\ti{B} | \ti{C})}$ are defined in Step 3.

The four axioms are about 2-isomorphisms in $\Hom(((\ti{A}\ti{B})\ti{C})\ti{D}, \ti{D}((\ti{A}\ti{B})\ti{C})),$ $\Hom(\ti{A}((\ti{B}\ti{C})\ti{D}), ((\ti{B}\ti{C})\ti{D})\ti{A}),$ $\Hom((\ti{A}\ti{B})(\ti{C}\ti{D}), (\ti{C}\ti{D})(\ti{A}\ti{B}))$ and
$\Hom((\ti{A}\ti{B})\ti{C}, \ti{C}(\ti{B}\ti{A}))$, respectively.
The first one follows from the definition of $R_{a,-}$ in the associator $\ti{a}: (\ti{A}\ti{B})\ti{C} \ra \ti{A}(\ti{B}\ti{C})$ as in (\ref{def ABCX}).
The second is the same as the axiom in (\ref{def AXYZ}).
The third follows from the definition of $R_{AB,-}$ in the tensor product $\ti{A}\ti{B}$ as in (\ref{def ABXY}).
The last one follows from the definition of $R_{R_{A,B},-}$ in the braiding $R_{\ti{A}, \ti{B}}$ as in (\ref{def ABX}).
\end{proof}

\section{Computation of $\zv$} \label{sec zv}

A monoidal bicategory is defined as a tricategory with one object.
We refer to \cite[Section 4.1]{Gur} for the definition of tricategories.
A monoidal 2-category is a monoidal bicategory whose underlying bicategory is a 2-category.

Let $\cv$ be the 1-category of finite dimensional $\K$-vector spaces (i.e. $\mathrm{1Vec}$).
Let $\ve$ be the 2-category of 1-categories of finite semisimple $\cv$-module categories \cite{Os1}. More precisely, objects in $\ve$ are of the form $\cv^{\bo n}$, where $\bo$ is the direct sum; 1-morphisms are the $\cv$-module functors; 2-morphisms are $\cv$-module natural transformations.
The only simple object is $\cv$ whose endomorphism 1-category $\End(\cv) \cong \cv$.
The tensor product $\bt$ in $\ve$ is the Deligne tensor product.

Consider the monoidal 2-category $(\vgw, \bt, I, \bf{a}, \bf{l}, \bf{r}, \pi, \mu, \rho, \lm)$.
It is isomorphic to a direct sum of $|G|$ copies of $\ve$ as 2-categories.
The simple objects are $\de_g$ for $g \in G$.
Any object is of the form $A=\bo_{g \in G}A_g$, where $A_g \in \ve$ is the $g$-component.

Tensor product of two simple objects is $\de_g \bt \de_g'=\de_{gg'}$.
The unit object $I=\de_1$.
The adjoint equivalences $\bf{a}, \bf{l}, \bf{r}$ are all identities (i.e. $a,l,r$ and the 2-isomorphisms defining their naturalities are all identities).
The invertible modifications $\rho$ and $\lm$ are determined by $\pi, \mu$ and the axioms.
So the monoidal structure is completely determined by $\pi$ and $\mu$.
Moreover, $\pi$ is described by a cocycle $\om \in Z^4(G, \Ks)$:
$$\xymatrix{
((\de_{x}\de_{y})\de_{z})\de_{w} \ar[r]^{=} \ar[d]^{=} &   (\de_{x}(\de_{y}\de_{z}))\de_{w} \ar[r]^{=} \ar@{}[d]|{\Downarrow \pi=\om(x,y,z,w)}& \de_{x}((\de_{y}\de_{z})\de_{w}) \ar[d]^{=} \\
(\de_{x}\de_{y})(\de_{z}\de_{w}) \ar[rr]_{=} && \de_{x}(\de_{y}(\de_{z}\de_{w}))
}$$
and $\mu$ is described by a 2-cochain in $C^2(G, \Ks)$ which satisfies certain compatibility conditions with $\om$.
We restrict ourself to the {\em normalized} case: (1) $\om$ is a normalized cocycle, i.e. $\om(x_1,x_2,x_3,x_4)=1$ if $x_i=1$ for some $i$; and (2) the 2-cochain $\mu$ is trivial, i.e. $\mu(x_1, x_2)=1$ for all $x_i$.
In this case, $\mu, \rho, \lm$ are all trivial so that the unit is strict. In particular, it means that the invertible modifications defined by (\ref{diag:unit-1}),(\ref{def Rl}),(\ref{diag:unit-3}) are all identities.
As a consequence, the diagram (\ref{diag:ABX}) is automatically commutative. 

\begin{rmk}
It is expected that isomorphism classes of monoidal structures on $\ve_G$ are classified by $H^4(G, \Ks)$.
Any class in $H^4(G, \Ks)$ has a normalized representative.
So our restriction to the normalized case is inessential.
\end{rmk}


\subsection{The 2-category}
We first compute $\zv$ as a 2-category.
Let $\ta=(A, R_{A,-}, \tr_{(A|-,?)})$ be an object of $\zv$.
The half braiding $R_{A,-}$ gives an equivalence of categories $R_{A,g}: A \bt \de_g \ra \de_g \bt A$, for any $g \in G$.
Since $\vgw$ is a 2-category, $R_{A,-}(id_{\de_g})=id_{R_{A,g}}$.
The equation $R_{A, X \bo Y}=R_{A,X} \bo R_{A,Y}$ implies that $R_{A,-}$ is completely determined by the collection $\{R_{A,g}\}$.

Let $\cl$ denote the set of conjugacy classes of $G$.
We write $h \in c$ and $[h]=c$ if $h \in G$ is in a conjugacy class $c \in \cl$.
Any object of $\zv$ has a direct sum decomposition $\ta=\bo_{c\in \cl}\ta_c$ into its $c$-components due to the half braiding.
It induces a decomposition $\zv=\bo_{c\in \cl}\zv_c$ of the 2-category.

\subsubsection{The component $\zv_c$}
We give an explicit description of one component $\zv_c$ following Step 1 in Section \ref{sec:def-center}.
Let $\{h_1, \dots, h_s\}$ denote all elements of $G$ in the class $c$.

\vspace{.2cm}
\n {\bf Objects.} For an object $\ta_c=(A_c, R_{A,-}, \tr_{(A|-,?)})$, its underlying object $A_c=\bo_{i}~ A_{h_i}$ in $\ve$.
The half braiding is a collection of equivalences
$$R_{A,g}=\bo_i ~ R_{h_i,g}, \quad R_{h_i, g}: A_{h_i} \de_g \ra \de_g A_{h_j},$$
for $h_ig=gh_j$.
The invertible modification $R_{(A|g,g')}=\bo_i~ R_{(h_i|g,g')}$:
\begin{equation} \label{eq new 0mor}
\xymatrix{
A_{h_i}\de_g \de_{g'} \ar[rrrr]^{R_{h_i,gg'}} \ar[drr]_{R_{h_i,g}} & & \ar@{}[d]^{\Downarrow R_{(h_i|g,g')}} & & \de_g \de_{g'} A_{h_k} \\
& & \de_g A_{h_j} \de_{g'} \ar[urr]_{ R_{h_j,g'}} & &
}\end{equation}
for $h_ig=gh_j, h_jg'=g'h_k$.
Here we omit 1-associators which are all identities.
The modifications $\tr_{(h_i|gg',g'')}, \tr_{(h_i|g,g')},\tr_{(h_j|g',g'')}, \tr_{(h_i|g,g'g'')}$ together with the 4-cocycle $\pi$ should satisfy the axiom in (\ref{def AXYZ}), for $A=A_{h_i}, X=\de_g, Y=\de_{g'}, Z=\de_{g''}$.
All adjoint equivalences $a$ are identities so that the four isomorphisms `$\cong$' are identities.
This axiom gives an equation of the 2-isomorphisms:
\begin{equation} \label{eq new key 2iso}
\tr_{(h_i|g,g'g'')}\cdot \tr_{(h_j|g',g'')}=\tau_{h_i}(\om)(g,g',g'') \cdot \tr_{(h_i|gg',g'')}\cdot \tr_{(h_i|g,g')},
\end{equation}
for $h_ig=gh_j, h_jg'=g'h_k, h_kg''=g''h_l$, and
$$\tau_{h_i}(\om)(g,g',g'')=\frac{\om(h_i,g,g',g'')\om(g,g',h_k,g'')}{\om(g,h_j,g',g'')\om(g,g',g'',h_l)}.$$
We introduce a handy notation for Equation\,(\ref{eq new key 2iso}): Eq$(h_i|g,g',g'')$. It is a consequence of the axiom in (\ref{def AXYZ}), which can be simplified by omitting 1- and 2-associators as follows:
\begin{equation} \label{eq new 0mor comm}\xymatrix{
A_{h_i} \de_g \de_{g'} \de_{g''} \ar[rr]^{R_{h_i,gg'g''}} \ar[d]_{R_{h_i,g}} \ar[drr] &  & \de_g \de_{g'} \de_{g''} A_{h_l} \\
 \de_g A_{h_j} \de_{g'} \de_{g''} \ar[rr]_{R_{h_j,g'} } \ar@{-->}[urr] & & \de_g \de_{g'} A_{h_k} \de_{g''}. \ar[u]_{R_{h_k,g''}}
}\end{equation}

\vspace{.2cm}
\n {\bf 1-morphisms.} A 1-morphism is $(f, R_{f,-}): (A_c, R_{A,-}, \tr_{(A|-,?)}) \ra (A_c', R'_{A',-}, \tr'_{(A'|-,?)})$ consists of a 1-morphism $f=\bo_i~f_i, f_i: A_{h_i} \ra A_{h_i}'$, and an invertible modification $R_{f,g}=\bo_i ~ R_{f_i,g}$:
$$\xymatrix{
A_{h_i}' \de_g \ar[r]^{R'_{h_i,g}} & \de_g A_{h_j}'\\
A_{h_i} \de_g \ar[u]^{f_i} \ar[r]_{R_{h_i,g}}  \ar@{}[ur]|{\Rightarrow R_{f_i,g}} &   \de_g A_{h_j}. \ar[u]_{f_j}
}$$
The invertible modifications $R_{f_i,g}, R_{f_j,g'}, R_{f_i,gg'}$ should satisfy the axiom in (\ref{def 1-mor}) for $A=A_{h_i}, A'=A'_{h_i}, X=\de_g, Y=\de_{g'}$.
The axiom is simplified to the following diagram by omitting identity 1-associators:
\begin{equation} \label{eq new 1mor comm}
\xymatrix{
A'_{h_i}\de_g \de_{g'} \ar[rrrr] \ar[drr]  & & \ar@{}[d]^{\Downarrow R'_{(h_i|g,g')}} & & \de_g \de_{g'} A'_{h_k} \\
\ar@{}[dr]^{\Rightarrow R_{(f_i,g)}} & & \de_g A'_{h_j} \de_{g'}  \ar@{}[dr]^{\Rightarrow R_{(f_j,g')}} \ar[urr]& & \\
A_{h_i}\de_g \de_{g'}   \ar[uu]^{f_i} \ar@{-->}[rrrr] \ar[drr] & & \ar@{}[d]^{\Downarrow R_{(h_i|g,g')}} & & \de_g \de_{g'} A_{h_k} \ar[uu]^{f_k}\\
& & \de_g A_{h_j} \de_{g'}  \ar[uu]^>>>>>>{f_j} \ar[urr] & &
}
\end{equation}
where the 2-isomorphism in the back is $R_{f_i,gg'}$.
We denote this compatibility condition for 1-morphisms as Eq1$(h_i|g,g')$.

\n{\bf 2-morphisms.} A 2-morphism $\alpha: (f, R_{f,-}) \Ra (f', R_{f',-})$ is a 2-morphism $\al=\bo_i~\al_i, \al_i: f_i \Ra f'_i$
which satisfies the axiom in (\ref{def 2-mor}) for $A=A_{h_i}, A'=A'_{h_i}, X=\de_g$:
\begin{equation} \label{eq new 2mor comm}
\al_j \cdot R_{f_i,g}=R_{f'_i, g} \cdot \al_i.
\end{equation}
We denote this compatibility condition for 2-morphisms as Eq2$(h_i|g)$.

\subsubsection{The restriction to one grading}
For an object $\ta_c$, its underlying object $A_c=\bo_{h \in c} A_h$ in $\ve$, where $A_h$ are all equivalent to each other by the requirement of the half braiding.
We pick up a grading $h \in c$, and let $\cgh=\{g \in G | gh=hg\}$ denote the centralizer of $h$ in $G$.
We focus on the component $A_h$ and the half braiding with $\de_x$ for $x \in \cgh$ in the following.

For $x \in \cgh$, the equivalence $R_{h,x}: A_h \de_x \ra \de_x A_h$ induces an autoequivalence of $A_h$:
$$\rho_x: A_h \ra A_h  \de_x \xra{R_{h,x}} \de_x  A_h \ra A_h,$$
where the first and last maps are grading shifts in $\vgw$ which are identities in $\ve$.

For $x,y \in \cgh$, the 2-modification $\tr_{(h|x,y)}$ as in (\ref{eq new 0mor}) induces a 2-isomorphism $m(x,y): \rho_y\rho_x \Rightarrow \rho_{xy}$ by taking the natural grading shifts to $A_h$.
Thus, the collection $\{\rho_x ~|~ x \in \cgh\}$ gives a weak right action of $\cgh$ on the 1-category $A_h$.

For $x,y,z \in \cgh$, the modifications $\tr_{(h|xy,z)}, \tr_{(h|x,y)},\tr_{(h|y,z)}, \tr_{(h|x,yz)}$ satisfy Eq$(h|x,y,z)$:
\begin{equation} \label{eq key 2iso}
\tr_{(h|x,yz)}\cdot \tr_{(h|y,z)}=\frac{\om(h,x,y,z)\om(x,y,h,z)}{\om(x,h,y,z)\om(x,y,z,h)}\tr_{(h|xy,z)}\cdot \tr_{(h|x,y)}.
\end{equation}
Note that $h_i=h_j=h_k=h_l=h$ in this case.
Translating to the weak action of $\cgh$ on $A_h$, the 2-isomorphisms satisfy the following equation:
\begin{equation} \label{eq 2iso}
m(x,yz)\cdot m(y,z) =\frac{\om(h,x,y,z)\om(x,y,h,z)}{\om(x,h,y,z)\om(x,y,z,h)} m(xy,z)\cdot m(x,y).
\end{equation}
The action is associative up to a twisting determined by $\om \in Z^4(G, \Ks)$.

Consider the transgression map $\tau_h: C^{k+1}(G, \Ks) \ra C^k(\cgh, \Ks)$ defined by:
$$\tau_h(\om)(x_1,\dots,x_k)=\prod_{0\le i \le k}\om(x_1,\dots,x_i,h,x_{i+1},\dots,x_k)^{(-1)^{i}},$$
for $x_i \in \cgh$.
It is straightforward to check that $\tau_h$ is a chain map.
It induces a map between cohomologies which is still denoted by $\tau_h$.
We are mainly interested in the case of $k=3$.

Equation (\ref{eq 2iso}) can be rewritten as
$$m(x,yz) \cdot m(y,z) =\tau_h(\om)  m(xy,z) \cdot m(x,y).$$
It follows that $A_h \in \rh$, i.e. it is a right module category over the monoidal 1-category $\vc_{\cgh}^{\tau_h(\om)}$.
So there is a forgetful map $\zv_c \ra \rh$ by taking its $h$-component.

On the level of morphisms, a 1-morphism $(f, R_{f,-})$ restricts to a collection $\{R_{f,x}: A_h \de_x \ra A'_h \de_x ~|~ x \in \cgh\}$ of 2-isomorphisms.
This collection defines a 1-morphism in $\rh$.
Similarly, 2-morphisms in $\zv$ restricts to 2-morphisms in $\rh$.
To sum up, we have a forgetful 2-functor $$\Phi_h: \zv_c \ra \rh$$ by restricting to the $h$-component.

\subsubsection{The equivalence of the forgetful functor}
We show that the forgetful functor $\Phi_h$ is an equivalence of 2-categories in the following.
Fix a set of representatives $\{g_i \in G ~|~ i=1,\dots,s ~\mbox{and}~ g_1=1\}$ for the coset $\cgh \backslash G$.
Then $\{h_i=g_i^{-1}hg_i ~|~ i=1,\dots,s\}$ are all elements in $c$, and $h_1=h$ is the base point.
We construct a 2-functor $\Psi_h: \rh \ra \zv_c$ in the inverse direction by extending the action of $\cgh$ on $A_h$ to that of $G$ on $A_c$.

\vspace{.2cm}
\n{\bf Step 1: Objects.}
Let $M=(M, \rho_x, m(x,y))$ be an object of $\rh$, where $\rho_x$ is the action and $m(x,y)$ is the 2-modification.
We want to extend $\rho_x, m(x,y)$ from the $h$-component to $h_i$-component via the path determined by $g_i$.
Define $\Psi_h(M)=(M_c, R_{M,-}, R_{(M|-,?)})$ as
$$M_c=\bo_{i}~M_{h_i}, \quad M_{h_i}=M,$$
$$R_{M,g}=\bo_i~R_{h_i,g}, \quad R_{h_i,g}: M_{h_i} \de_g \xra{=} M \xra{\rho_x} M \xra{=} \de_g M_{h_j},$$
where given $i, j$ and $g \in G$, there is a unique $x \in \cgh$ such that $g_ig=xg_j$.
The 2-modification $R_{(M|g,g')}=\bo_i R_{(h_i|g,g')}$, and $R_{(h_i|g,g')}$ is defined in the following order: $$R_{(h|x,y)}, R_{(h|x,g_i)}, R_{(h|g_i,g)}, R_{(h|x,g)}, R_{(h|g,g')}, R_{(h_i|g,g')},$$ where $x, y \in \cgh$ and $g,g' \in G$.
The initial data is to define $R_{(h|x,y)}=m(x,y)$ and choose any 2-isomorphisms for $R_{(h|x,g_i)}, R_{(h|g_i,g)}$ only requiring that $R_{(h|x,1)}=R_{(h|1,g)}=R_{(h|1,1)}$.
Eq$(h|x,y,g_i)$ involves four 2-isomorphisms:
$$\tr_{(h|x,yg_i)}, \tr_{(h|y,g_i)}, \tr_{(h|x,y)}, \tr_{(h|xy, g_i)}.$$
So $R_{(h|x,g)}$ for $g=yg_i$ is determined by the other three isomorphisms which are already given.
Similarly, Eq$(h|x,g_i,g')$ uniquely determines $R_{(h|g,g')}$ for $g=xg_i$, and Eq$(h|g_i,g,g')$ uniquely determines $R_{(h_i|g,g')}$.

\begin{lemma} \label{lem psi obj}
The construction $(M_c, R_{M,-}, R_{(M|-,?)})$ gives a well-defined object of $\zv_c$.
\end{lemma}
\begin{proof}
By definition it suffices to show that Eq$(h_i|g,g',g'')$ in (\ref{eq new key 2iso}) holds for all $g,g',g'' \in G$ and all $i=1,\dots,s$.
The key point is that there is a compatibility condition between Equations $$\mbox{Eq}(h_i|g,g',g''), \mbox{Eq}(h_i|g,g',g''g'''), \mbox{Eq}(h_i|g,g'g'',g'''),\mbox{Eq}(h_i|gg',g'',g'''),\mbox{Eq}(h_j|g',g'',g''')$$ from the axiom (\ref{def AXYZ}) for $M_{h_i}  \de_g \de_{g'}  \de_{g''} \de_{g'''}$, where $h_ig=gh_j$.
We denote this compatibility condition by CC$(h_i|g,g',g'',g''')$.
If any four of the five equations hold then so is the remaining one.
We prove that Eq$(h_i|g,g',g'')$ holds in the following order: (1) $(h|x,y,z), (h|x,y,g_i), (h|x,g_i,g), (h|g_i,g,g')$, and (2) $(h|x,y,g), (h|x,g,g'),$ $(h|g,g',g''), (h_i|g,g',g'')$, where $x,y,z \in \cgh, g,g',g'' \in G$.
The equations in the first group holds from the construction.
The condition CC$(h|x,y,z,g_i)$ implies that Eq$(h|x,y,g)$ holds for $g=zg_i$ since the other four equations Eq$(h|x,y,z)$, Eq$(h|xy,z,g_i)$, Eq$(h|x,yz,g_i)$, Eq$(h|y,z,g_i)$ hold.
Similarly, the condition CC$(h|x,y,g_i,g')$ implies that Eq$(h|x,g,g')$ holds for $g=yg_i$; CC$(h|x,g_i,g',g'')$ implies that Eq$(h|g,g',g'')$ holds for $g=xg_i$; and CC$(h|g_i,g,g',g'')$ implies that Eq$(h_i|g,g',g'')$ holds.
\end{proof}

\vspace{.2cm}
\n{\bf Step 2: 1-morphisms.}
Let $(f, M_{f,x}): (M, \rho_x, m(x,y)) \ra (M', \rho'_x, m'(x,y))$ denote a 1-morphism in $\rh$, where $f: M \ra M'$, and $M_{f,x}$ is the 2-modification for $x \in \cgh$.
We define $\Psi_h(f, R_{f,x})=\bo_i(f_i, R_{f_i,-}): \Psi_h(M) \ra \Psi_h(M')$, where $f_i: M_{h_i} \xra{=} M \xra{f} M' \xra{=} M'_{h_i}$, and $R_{f_i,g}$ is the 2-modification for $g \in G$ given below.

The only constraint for a 1-morphism is Eq1$(h_i|g,g')$ in (\ref{eq new 1mor comm}) for $A_{h_i}=M_{h_i}, A'_{h_i}=M'_{h_i}$.
Eq1$(h_i|g,g')$ contains five terms $R_{f_i,g}, R_{f_i, gg'}, R_{f_j, g'}$ and $R_{(h_i|g,g')}, R'_{(h_i|g,g')}$, where $h_ig=gh_j$, and the last two terms are already given.
For the first three terms, any two of them determines the remaining one.

We define $R_{f_i,g}$ in the following order: $R_{f_1,x}, R_{f_1,g_i}, R_{f_1,g}, R_{f_i,g}$ for $x \in \cgh, g \in G$.
Note that $h_1=h$ is the base point.
The initial data is to define $R_{f_1,x}=M_{f,x}$ and $R_{f_1,g_i}=id$ for all $i=1,\dots,s$.
Eq1$(h_1|x,g_i)$ implies that $R_{f_1, g}$ for $g=xg_i$ is uniquely determined by $R_{f_1,x}$ and $R_{f_1, g_i}$.
Eq1$(h_1|g_i,g)$ implies that $R_{f_i, g}$ is uniquely determined by $R_{f_1,g_i}$ and $R_{f_1, g_ig}$.

An argument similar to the proof of Lemma \ref{lem psi obj} shows that $\Psi_h(f, R_{f,x})=\bo_i(f_i, R_{f_i,-})$ gives a well-defined 1-morphism in $\zv_c$.
It suffices to show that Eq1$(h_i|g,g')$ holds for all $g,g'\in G$.
There is a compatibility condition between $$\mbox{Eq}1(h_i|g,g'), \mbox{Eq}1(h_i|g,g'g''), \mbox{Eq}1(h_i|gg',g''),\mbox{Eq}1(h_j|g',g'')$$ from (\ref{eq new 1mor comm}) for $M_{h_i}  \de_g  \de_{g'}  \de_{g''}$.
We denote this compatibility condition as CC1$(h_i|g,g',g'')$.
If any three of the four constraints hold then so is the remaining one.
We prove that Eq1$(h_i|g,g')$ holds in the following order: (1) $(h_1|x,y), (h_1|x,g_i), (h_1|g_i,g)$, and (2) $(h_1|x,g), (h_1|g,g'), (h_i|g,g'), (h_i|g,g',g'')$, where $x,y \in \cgh, g,g' \in G$.
The constraints in the first group holds from the construction.
The condition CC1$(h_1|x,y,g_i)$ implies that Eq1$(h_1|x,g)$ holds for $g=yg_i$; CC1$(h_1|x,g_i,g')$ implies that Eq1$(h_1|g,g')$ holds for $g=xg_i$; and CC1$(h_1|g_i,g,g')$ implies that Eq1$(h_i|g,g')$ holds.

\vspace{.2cm}
\n{\bf Step 3: 2-morphisms.}
Let $\al: (f, M_{f,x}) \Ra (f', M_{f',x})$ be a 2-morphism in $\rh$.
We define $\Psi_h(\al)=\bo_i \al_i: \Psi_h(f, M_{f,x}) \Ra \Psi_h(f', M_{f',x})$, where $\al_i: f_i \Ra f'_i$ is given below.
The only constraint for a 2-morphism is Eq2$(h_i|g)$ in (\ref{eq new 2mor comm}).
The term $\al_j$ is determined by $\al_i$ since $R_{f_i,g}$ and $R_{f'_i,g}$ are isomorphisms.

We define $\al_1=\al$ as the 2-morphism in $\rh$, and define $\al_i$ from $\al_1$ and Eq2$(h_1|g_i)$ for $i=2,\dots,s$.
A similar argument shows that $\Psi_h(\al)=\bo_i \al_i$ gives a well-defined 2-morphism in $\zv_c$.

We complete the definition of the 2-functor $\Psi_h: \rh \ra \zv_c$.

\vspace{.2cm}
To show that $\Phi_h$ and $\Psi_h$ give an equivalence of 2-categories, it is obvious that $\Phi_h \circ \Psi_h$ is the identity 2-functor.
It remains to show that $\Psi_h$ is essentially surjective and fully faithful.
The proof is similar to the construction of $\Psi_h$ above and we leave it to the reader.

\begin{thm} \label{thm1}
There is an equivalence of 2-categories:
$$
\zv \simeq \bo_{[h] \in \cl} \rh,
$$
by choosing one representative $h$ for each class $c \in \cl$.
In particular, $\zv$ is semisimple in the sense of Douglas and Reutter \cite{DR}.
\end{thm}

Any object $\ta_c$ of $\zv$ is determined by one of its component $A_h$ as an object of $\rh$ from Theorem \ref{thm1}.
It is known that any indecomposable object of $\rh$ is given by a pair $(H, \psi)$, where $H$ is a subgroup of $\cgh$, $\psi \in C^2(H, \Ks)$ such that $d \psi=\tau_h(\om)^{-1}|_H$ \cite[Example 2.1]{Os2}.
Note that we consider right modules over $\vc_{\cgh}^{\tau_h(\om)}$ instead of left modules.
More precisely, the object associated to $(H, \psi)$ is $\bo_{s \in H \backslash \cgh}\cv(s)$, where each component $\cv(s)=\cv$.
The action of $\vc_{\cgh}^{\tau_h(\om)}$ is given by multiplication in $\cgh$ on the right.
The stablizer of $\cv(1)$ is equivalent to $\vc_H$, and $\psi$ determines its 1-associator.
Let $\cv(H \backslash K)=\bo_{s \in H \backslash K}\cv(s)$ for $H<K$.

We express any indecomposable object $\ta_c$ as
$$A(h,H,\psi), \quad \mbox{where}~~[h]=c, H < \cgh,  \psi \in C^2(H, \Ks), d \psi=\tau_h(\om)^{-1}|_H.$$
The presentation is independent of the choice of $h \in c$: $A(h, H, \psi) \simeq A(\gh, g^{-1}Hg, g^*(\psi))$, where $g^*(\psi) \in C^2(g^{-1}Hg, \Ks)$ is induced by conjugation.

We discuss a simple example where 2-categories $\zv$ could be inequivalent for different classes $\om$.

\begin{example} \label{ex Z2Z2}
Consider an abelian group $G=\Z_2 \oplus \Z_2=\{1,s_1,s_2,s_1s_2\}$.
There is a canonical isomorphism $f: H^{k+1}(G, \Z) \ra H^{k}(G, \Ks)$ for $k \ge 1$ which commutes with the transgression map $\tau_h$.
The cup product makes $H^*(G, \Z)$ a graded super commutative ring as
$$H^*(G, \Z) \cong \Z_2[\alpha, \beta, \gamma] / (\gamma^2-\alpha\beta(\alpha+\beta)), \quad \deg(\alpha)=\deg(\beta)=2, \deg(\gamma)=3,$$
see \cite[Proposition 4.1]{Le}.
We have isomorphisms of abelian groups:
$$H^2(G,\Z)=\Z_2 \lan \alpha, \beta \ran, H^3(G,\Z)=\Z_2\lan \gamma \ran, H^4(G,\Z)=\Z_2\lan \alpha^2, \alpha\beta, \beta^2 \ran, H^5(G,\Z)=\Z_2\lan \alpha\gamma, \beta\gamma \ran.$$
The transgression map $\tau_h: H^{k+1}(G, \Z) \ra H^{k}(G, \Z)$ is a derivation $\tau_h(ab)=\tau_h(a)b+a\tau_h(b)$, where the signs are irrelevant since all groups are $2$-torsion.
By properly choosing generators $\alpha, \beta$, we could have $\tau_1(\gamma)=0, \tau_{s_1}(\gamma)=\alpha, \tau_{s_2}(\gamma)=\beta, \tau_{s_1s_2}(\gamma)=\alpha+\beta$, and $\tau_h(\alpha)=\tau_h(\beta)=0$, for all $h \in G$.
So
$$\tau_1(\alpha\gamma)=0, \tau_{s_1}(\alpha\gamma)=\alpha^2, \tau_{s_2}(\alpha\gamma)=\alpha\beta, \tau_{s_1s_2}(\alpha\gamma)=\alpha^2+ \alpha\beta.$$

Consider two classes $\om_0=1, \om_1=f(\alpha\gamma) \in H^4(G, \Ks)$.
By Theorem \ref{thm1}, we obtain the following equivalences of 2-categories:
\begin{align*}
\mathcal{Z}(\mathrm{2Vec}_{G}^{\omega_0}) &\simeq \rep(G)^{\bo 4}, \\
\mathcal{Z}(\mathrm{2Vec}_{G}^{\omega_1}) &\simeq \rep(G) \bo \rep(G, f(\alpha^2)) \bo \rep(G, f(\alpha\beta)) \bo \rep(G, f(\alpha^2+ \alpha\beta)).
\end{align*}
The equivalence classes of indecomposable objects of $\rep(G, \chi)$ for $\chi \in H^3(G, \Ks)$ are classified by the conjugacy classes of pairs $(H, \psi)$ where $H<G$ is a subgroup such that $\chi|_H=1$, and $\psi \in H^2(H, \Ks)$, see \cite[Example 2.1]{Os2}.
The number of equivalence classes of indecomposable objects of $\rep(G, \chi)$ is finite, denoted by $c(G, \chi)$.
Taking $H=G$, $\rep(G)$ has an indecomposable object $(G, \psi)$ for $\psi \in H^2(G, \Ks)$, while $\rep(G, \chi)$ does not have such an indecomposable object when $\chi$ is nontrivial.
So $c(G,1) > c(G, \chi)$, and the 2-categories $\rep(G)$ and $\rep(G, \chi)$ are not equivalent for nontrivial $\chi$.
Hence, $\mathcal{Z}(\mathrm{2Vec}_{G}^{\omega_0})$ and $\mathcal{Z}(\mathrm{2Vec}_{G}^{\omega_1})$ are not equivalent as 2-categories.
\end{example}

\subsection{The braided monoidal 2-category} \label{sec:braided zv}
Before we compute the tensor product $A(h,H,\psi)\bt A(h',H',\psi')$, we first forget about the grading.  We have $A(h,H,\psi) \simeq \cv(H \backslash G)$ as objects in $\ve$.
The half braiding induces a weak action of $\vc_G$ on $\cv(H \backslash G)$ which is given by multiplication in $G$ on the right.
The tensor product of two weak right $\vc_G$ module categories is given by the Deligne tensor product, and we have
\begin{equation} \label{eq HH'}
\cv(H \backslash G) \bt \cv(H' \backslash G) \cong \bo_{t \in H \backslash G / H'} \cv(H_t \backslash G),
\end{equation}
where the sum is over the double coset $H \backslash G / H'$, and $H_t=t^{-1}Ht \cap H'$.

A direct computation from (\ref{def ABXY}) shows that $A(h,H,\psi)\bt A(h',H',\psi')$ contains a component $A(h_t,H_t,\psi_t)$,
where $t \in H \backslash G / H'$, $h_t=t^{-1}hth', H_t=t^{-1}Ht \cap H'$, and
\begin{equation} \label{eq psit}
\psi_t=t^*(\psi)|_{H_t}\cdot \psi'|_{H_t}\prod\limits_{0 \le i \le j \le 2}\psi_{ij,t}^{(-1)^{i+j}}\in C^2(H_t, \Ks).
\end{equation}
Here $\psi_{ij,t}(x_1, x_2)=\om(\dots, x_i, t^{-1}ht, \dots, x_j, h', \dots),$ for $0 \le i \le j \le 2$. The underlying 2-category of $A(h_t,H_t,\psi_t)$ is precisely $\cv(H_t \backslash G)$ in (\ref{eq HH'}).

\begin{lemma} \label{lem tau hh'}
Given $A(h,H,\psi), A(h',H',\psi')$ and $t \in H \backslash G / H'$, we have $H_t < C_G(t^{-1}hth')$ and $d \psi_t=\tau_{t^{-1}hth'}(\om)^{-1}|_{H_t}$.
\end{lemma}
\begin{proof}
We only check the case of $t=1$.
We have $H_1=H \cap H' < \cgh \cap C_G(h') < C_G(hh')$.
Consider the trivial cochain $\chi_{kl}=1 \in C^3(H\cap H', \Ks): \chi_{kl}(x_1,x_2,x_3)=d\om(\dots, x_k, h, \dots, x_l, h', \dots),$
for $0 \le k \le l \le 3$.
A direct computation shows that
$$\tau_{hh'}(\om)\tau_{h}(\om)^{-1}\tau_{h'}(\om)^{-1}\prod\limits_{0 \le i \le j \le 2}d\psi_{ij}^{(-1)^{i+j}}=\prod\limits_{0 \le k \le l \le 3}\chi_{kl}^{(-1)^{k+l}}=1,$$
when restricting to $H \cap H'$.
So $d \psi_1=\tau_{hh'}(\om)^{-1}|_{H \cap H'}$.
\end{proof}

\begin{prop} \label{prop tensor}
The tensor product of two indecomposable objects in $\zv$ is given by $$A(h,H,\psi) \bt A(h',H',\psi') \cong \bo_{t}A(h_t,H_t,\psi_t),$$
where the sum is over $t \in H \backslash G / H'$.
\end{prop}
\begin{proof}
Lemma \ref{lem tau hh'} implies that $A(h_t,H_t,\psi_t)$ is well-defined.
Each component of the right hand side appears in the tensor product at least once.
It follows from (\ref{eq HH'}) that each of them appears at most once.
\end{proof}

The 1-associators in $\vgw$ are all identities.
In the contrast, a 1-associator $\ti{a}: (\ti{A} \ti{B}) \ti{C} \ra \ti{A}(\ti{B}\ti{C})$ is a 1-morphism $(a, R_{a,-})$, where $a: (AA')A'' \ra A(A'A'')$ is the identity in $\vgw$, and $R_{a,-}$ is an invertible modification given by Diagram (\ref{def ABCX}) which might be nontrivial.
The associators $\ti{l}, \ti{r}$ are all identities since the 4-cocycle $\om$ is normalized.

Invertible modifications $\tilde{\pi}, \tilde{\mu}, \tilde{\lm}, \tilde{\rho}$ are defined in the same way as in $\vgw$.
In particular, $\tilde{\mu}, \tilde{\lm}, \tilde{\rho}$ are all identities, and $\tilde{\pi}$ is given by $\om$.

\medskip

The braiding of two objects $\ti{A}=(A, R_{A,-}, \tr_{(A|-,?)})$ and $\ti{B}=(B, R'_{B,-}, \tr'_{(B|-,?)})$ is a 1-morphism $R_{\ti{A}, \ti{B}}=(R_{A, B}, R_{R_{A, B}, -}): \ti{A} \ti{B} \ra \ti{B} \ti{A}$ in $\zv$, where $R_{A,B}=R_{A,-}(B): AB \ra BA$ is determined by the half braiding associated to $\ti{A}$ and the grading of $B$, and $R_{R_{A, B}, -}$ is an invertible modification given in Diagram (\ref{def ABX}).
More precisely, $R_{A,B}=\bo~R_{h_i,g}:$
\begin{equation} \label{eq braid}
\begin{array}{cccc}
R_{h_i,g}: & A_{h_i}B_g & \ra & B_gA_{h_j}  \\
& (x,y) & \mapsto & (y, \rho_g(x)),
\end{array}\end{equation}
for $x \in A_{h_i}, y \in B_g$, and $\rho_g: A_{h_i} \ra A_{h_j}$ is the action of $G$ on $A$ for $h_ig=gh_j$.

When $B$ is concentrated in the grading $1$, we have
\begin{equation} \label{eq swap}
R_{A,B}=\Sigma_{A, B}: AB \ra BA
\end{equation}
where $\Sigma_{A,B}$ is the canonical permutation equivalence between the Deligne tensor products which simply permutes the two factors $A$ and $B$ as objects of $\ve$.

\begin{rmk} \label{rmk RAf}
The naturality 2-isomorphism $R_{A,f}$ associated to a 1-morphism $f: B \ra B'$ is the identity when $B$ and $B'$ are concentrated in the grading $1$.
\end{rmk}

The invertible modifications $\tr_{(\ti{A}|\ti{B}, \ti{C})}=\tr_{(A|-,?)}(B,C)=\tr_{(A|B,C)}$ is given by the half braiding associated to $\ti{A}$,
and $\tr_{(\ti{A},\ti{B} | \ti{C})}$ is the identity as in Diagram (\ref{def R(A,B|C)}) since the 1-associators are the identities.

\medskip
In summary, $\zv$ is a braided monoidal 2-category whose underlying 2-category is given in Theorem\,\ref{thm1},  and the monoidal structure is given by Proposition\,\ref{prop tensor}, and the braiding structure is given by the half-braidings as explained in Step 3 in Section\,\ref{sec:def-center}.

\begin{example} \label{ex Z2}
Consider $G=\Z_2=\{1,s\}, \om=1$. There are two conjugacy classes: $h=1, h=s$.
We have an equivalence $\mathcal{Z}(\mathrm{2Vec}_{\Z_2}^\omega)=\mathcal{Z}(\mathrm{2Vec}_{\Z_2}^\omega)_1 \bo \mathcal{Z}(\mathrm{2Vec}_{\Z_2}^\omega)_s \simeq \rep(\Z_2) \bo \rep(\Z_2)$ of 2-categories from Theorem \ref{thm1}.
Up to isomorphism $\rep(\Z_2)$ has two indecomposable objects: the unit $I$ and the regular representation $T=\vc_{\Z_2}$.
A complete set of isomorphism classes of indecomposable objects of $\mathcal{Z}(\mathrm{2Vec}_{\Z_2}^\omega)$ is $\{I, T, I_s, T_s\}$, where the subscript $s$ denotes the nontrivial grading.

The nontrivial 1-categories of 1-morphisms are
\begin{align*}\End(I) \simeq \op{Rep}(\Z_2), \End(T) \simeq \op{1Vec}_{\Z_2}, \Hom(I,T) \simeq \op{1Vec}, \Hom(T,I) \simeq \op{1Vec}; \\
\End(I_s) \simeq \op{Rep}(\Z_2), \End(T_s) \simeq \op{1Vec}_{\Z_2}, \Hom(I_s,T_s) \simeq \op{1Vec}, \Hom(T_s,I_s) \simeq \op{1Vec}.
\end{align*}
We illustrate these structures in the following quiver:
\begin{align} \label{quiver1}
\xymatrix{
I \ar@(ul,ur)[]^{\op{Rep}(\Z_2)}  \ar@/^/[rr]^{\op{1Vec}} & & T \ar@(ul,ur)[]^{\op{1Vec}_{\Z_2}} \ar@/^/[ll]^{\op{1Vec}}
& & I_s \ar@(ul,ur)[]^{\op{Rep}(\Z_2)}  \ar@/^/[rr]^{\op{1Vec}} & & T_s \ar@(ul,ur)[]^{\op{1Vec}_{\Z_2}} \ar@/^/[ll]^{\op{1Vec}}
}
\end{align}

For the monoidal structure, $I$ is the unit, and we have
\begin{align*}
I_s \bt I_s \cong I, \quad I_s \bt T \cong T \bt I_s \cong T_s, \quad
T \bt T \cong T_s \bt T_s \cong T \bo T, \quad T \bt T_s \cong T_s \bt T \cong T_s \bo T_s.
\end{align*}
from Proposition \ref{prop tensor}.
The braiding is given by
$$\begin{array}{cccc}
R_{X,Y}: & XY & \ra & YX  \\
& (x,y) & \mapsto & (y, \rho_g(x)),
\end{array}$$
where $g=1$ for $Y=I,T$, $g=s$ for $Y=I_s,T_s$, and $\rho_g: X \ra X$ is the action of $G$ on $X$.
If $X=T, T_s$ and $Y=I_s, T_s$, then $R_{X,Y} \neq \Sigma_{X,Y}$; otherwise $R_{X,Y}=\Sigma_{X,Y}$ from (\ref{eq swap}).
\end{example}

\subsection{The unit component} \label{sec unit}
The unit component $\zv_c$ for $c=1$ is a braided monoidal sub-2-category of $\zv$.
In this case, $h=1, \cgh=G$ and $\tau_1(\om)$ is a coboundary for any $\om \in Z^4(G, \Ks)$.
If $\om$ is normalized, then $\tau_1(\om)=1$.
So $\rep(C_G(1), \tau_1(\om))$ is equivalent to the 2-category $\rep(G)$ of module categories over $\op{1Vec}_G$.
In particular, $\zv_1\simeq \rep(G)$ as braided monoidal 2-categories.

\begin{cor} \label{cor mod}
There is an inclusion $\rep(G) \hookrightarrow \zv$ of braided monoidal 2-categories for any $\om \in Z^4(G, \Ks)$.
\end{cor}

The 2-category $\rep(G)$ is well studied in \cite{Os2}.
More precisely, any indecomposable object of $\rep(G)$ is given by a pair $A=A(H,\psi)$, where $H < G$ and $\psi \in Z^2(H, \Ks)$.
The isomorphism class of $A(H,\psi)$ is determined by the conjugacy class of $H$ and the cohomological class $[\psi] \in H^2(H, \Ks)$.
There are two distinguished objects of $\zv_1$: one is the unit $I=A(H,\psi)$ for $H=G, \psi=1$; the other one is $T=A(H,\psi)$ for $H=1, \psi=1$.
As objects of $\rep(G)$, $I = \cv$ is the trivial representation, and $T = \vc_G$ is the regular representation of $\vc_G$.
The endomorphism 1-categories are
$$\End(I) \simeq \op{Rep}(G), \quad \End(T) \simeq \op{1Vec}_G.$$

For indecomposable objects $M, N$ of $\rep(G)$, bimodules $\op{Hom}_{\rep(G)}(M,N)$ and $\op{Hom}_{\rep(G)}(N,M)$ induces the Morita equivalence between $\op{End}_{\rep(G)}(M)$ and $\op{End}_{\rep(G)}(N)$.
Thus, $\rep(G)$ is the idempotent completion of the delooping of $\op{Rep}(G)$ in the sense of Douglas and Reutter \cite{DR}.
We illustrate these structures in the following quiver which is connected.
\begin{align} \label{quiver2}
\xymatrix{
 & & A(H,\psi) \ar@(ul,ur)[]^{\op{End}_{\rep(G)}(A(H,\psi))} \ar@/^/[drr]   \ar@/^/[dll]^<<<<<{\op{Rep}(H,\psi)}& & \\
A(G,1)=I \ar@(dl,dr)[]_{\op{Rep}(G)}  \ar@/^/[urr]^{\op{Rep}(H,\psi^{-1})}\ar@{<->}[rrrr]_{\op{1Vec}} & & & & T=A(1,1) \ar@(dl,dr)[]_{\op{1Vec}_{G}} \ar@/^/[ull]
}
\end{align}

It follows from Proposition \ref{prop tensor} that $$A(H,\psi) \bt A(H',\psi') \cong \bo_{t}A(H_t,\psi_t),$$
where the sum is over $t \in H \backslash G / H'$, $H_t=t^{-1}Ht \cap H'$, and $\psi_t=t^*(\psi)|_{H_t}\cdot \psi'|_{H_t}$ from (\ref{eq psit}) since $\om$ is normalized.
In particular, $A(H,\psi) \bt T \cong T\bt A(H,\psi) \cong T^{\bo H \backslash G}$.
Moreover, the monoidal structure is strictly associative since the invertible modifications in Diagram (\ref{def ABCX}) are all identities.

The braiding $R_{\ti{A}, \ti{B}}=(R_{A, B}, R_{R_{A, B}, -}): \ti{A} \ti{B} \ra \ti{B} \ti{A}$, where $R_{A,B}=R_{A,-}(B)=\Sigma_{A,B}: AB \ra BA$ from (\ref{eq swap}) since $B$ is concentrated in grading 1, and the invertible modification $R_{R_{A, B}, -}$ in Diagram (\ref{def ABX}) is the identity.

The invertible modifications $\tr_{(\ti{A}|\ti{B}, \ti{C})}, \tr_{(\ti{A},\ti{B} | \ti{C})}$ are all identities.

\subsection{The sylleptic center} \label{sec muger}
We briefly discuss the sylleptic center of $\rep(G)$ and $\zv$.
Crans gave a definition of the sylleptic center of a braided monoidal 2-category in the semistrict case \cite[Section 5.1]{Cr}.
We need a weak version. We propose the following definition without checking the coherence.

\begin{defn} \label{def:m-center}
Let $\C$ be a braided monoidal 2-category, and let $\D$ be a full monoidal sub-2-category. The {\em sylleptic centralizer} of $\D$ in $\C$, denoted by $\mathrm{Z}_\C(\D)$, is a 2-category defined as follows:
\begin{enumerate}
\item An object in $\mathrm{Z}_\C(\D)$ is a pair $(A, v_{A,-})$, where $A$ is an object of $\C$, and $v_{A,-}$ is an invertible modification
\begin{equation} \label{diag:sylleptic}
\xymatrix{
AX \ar@{=}[rr] \ar[dr]_{R_{A,X}} & \ar@{}[d]|{\Downarrow v_{A,X}} & AX \\
& XA \ar[ur]_{R_{X,A}} &
}
\end{equation}
for all $X\in\D$ such that the following axiom holds for all $X,Y\in\D$:
\begin{equation} \label{def Muger}
\resizebox{\displaywidth}{!}{\xymatrix{
(AX)Y \ar@{=}[rr] \ar[d]_{R_{A,X}} & \ar@{}[d]|{\Downarrow v_{A,X}} & (AX)Y  & & (AX)Y \ar[dr]_{a} \ar@{=}[rrrr] \ar[d]_{R_{A,X}} & & & & (AX)Y\\
(XA)Y \ar@{=}[rr] \ar[d]_{a} & & (XA)Y \ar[u]_{R_{X,A}} & = & (XA)Y  \ar[d]_{a} & A(XY) \ar@{=}[rr] \ar[d]^{R_{A,XY}}& \ar@{}[d]|{\Downarrow v_{A,XY}} & A(XY) \ar[ur]_{a^*} & (XA)Y \ar[u]_{R_{X,A}} \\
X(AY) \ar@{=}[rr]  \ar[dr]_{R_{A,Y}} & \ar@{}[d]|{\Downarrow v_{A,Y}} & X(AY) \ar[u]_{a^*} & & X(AY)  \ar@{}[ur]|{\Rightarrow R_{(A|X,Y)}} \ar[drr]_{R_{A,Y}} & (XY)A \ar[dr]^{a} \ar@{=}[rr] &  & (XY)A \ar[u]^{R_{XY,A}}  \ar@{}[ur]|{\Leftarrow R_{(X,Y|A)}} & X(AY) \ar[u]_{a^*} \\
& X(YA) \ar[ur]_{R_{Y,A}} & & & & & X(YA) \ar[urr]_{R_{Y,A}} \ar[ur]^{a^*} & &
}}\end{equation}

Note that this is an equality between two 2-morphisms, each of which is a composition of 2-morphisms defined by above two diagrams.

\item A 1-morphism from $(A, v_{A,-})$ to $(A', v_{A',-})$ is a 1-morphism $f: A \rightarrow A'$ in $\C$ such that the following diagram commutes:
\begin{equation} \label{def Muger 1-mor}
\xymatrix{
A'X \ar@{=}[rr] \ar[dr]_{R_{A',X}} & \ar@{}[d]|{\Downarrow v_{A',X}} & A'X \\
& XA' \ar[ur]_{R_{X,A'}} & \\
AX \ar@{=}[rr] \ar[dr]_{R_{A,X}} \ar[uu] \ar@{}[ur]|{\Rightarrow R_{f,X}} & \ar@{}[d]|{\Downarrow v_{A,X}} & AX \ar[uu] \ar@{}[ul]|{\Rightarrow R_{X,f}} \\
& XA \ar[ur]_{R_{X,A}} \ar[uu] &
}\end{equation}
where all vertical arrows are induced by $f$, and the 2-isomorphism in the back is the identity 2-isomorphism.
\item A 2-morphism is defined in the same way as in $\C$.
\end{enumerate}
When $\D=\C$, $\mathrm{Z}_\C(\C)$ is called the {\em sylleptic center} of $\C$.
\end{defn}

The monoidal, the braiding and the syllepsis structures on $\mathrm{Z}_\C(\D)$ can be generalized from Crans' definition in a similar way. We omit the detail here.

\begin{prop} \label{prop Muger repG}
The sylleptic center of $\rep(G)$ is equivalent to $\rep(G)$ as 2-categories.
\end{prop}
\begin{proof}
Let $(A, v_{A,-})$ be an object of the sylleptic center of $\rep(G)$.
The braiding of $\rep(G)$ is symmetric, i.e. $R_{X, A} \circ R_{A, X}=id_{A X}$ for any $X$.
We prove that the modification $v_{A, X}=id_{id_{AX}}$ as follows.
Taking $X=Y=I$ in the axiom (\ref{def Muger}) gives $v_{A,I}^2=v_{A,I}$ since $R_{(A|X,Y)}, R_{(X,Y|A)}$ are identities.
It follows that $v_{A,I}$ is the identity.
The semisimple 2-category $\rep(G)$ is equivalent to the module category of $\op{1Rep}(G)$, i.e. the idempotent completion of the delooping of $\op{1Rep}(G)$. The 2-category $\rep(G)$ has only one connected component since $\op{1Rep}(G)$ is fusion, see \cite[Remark 2.1.22]{DR}.
In other words, there exists a nontrivial 1-morphism $f:I \ra X$ for any object $X$ of $\rep(G)$.
The naturality of $v_{A,-}$ associated to $f$ is described by the following diagram:
$$\xymatrix{
AX \ar@{=}[rr] \ar[dr] & \ar@{}[d]|{\Downarrow v_{A,X}} & AX \\
& XA \ar[ur] \ar@{}[dl]|{\Rightarrow R_{A,f}} \ar@{}[dr]|{\Rightarrow R_{f,A}} & \\
AI \ar@{=}[rr] \ar[dr] \ar[uu]^{f} & \ar@{}[d]|{\Downarrow v_{A,I}} & AI \ar[uu]^{f} \\
& IA. \ar[ur] \ar[uu]^>>>>>>{f} &
}$$
Then $R_{f,A}$ is the identity since $A$ is concentrated in the grading $1$, and $R_{A,f}$ is the identity from Remark \ref{rmk RAf}.
It follows that $v_{A, X}=id_{id_{AX}}$.
Therefore, an object $(A, v_{A,-})$ of the sylleptic center of $\rep(G)$ is completely determined by $A$ as an object of $\rep(G)$.
For 1-morphism from $(A, v_{A,-})$ to $(A', v_{A',-})$, any 1-morphism $f: A \ra A'$ in $\rep(G)$ satisfies Diagram (\ref{def Muger 1-mor}) since all 2-isomorphisms are identities there.
\end{proof}

\begin{rmk}
The 2-category $\rep(G)$ has a natural syllepsis structure viewed as the sylleptic center of $\rep(G)$.
This syllepsis structure is {\em symmetric} in the sense of Crans \cite{Cr}, i.e. $id_{R_{X,Y}} \cdot v_{X,Y}=v_{Y,X} \cdot id_{R_{X,Y}}$ as 2-morphisms from $R_{X,Y} \circ R_{Y,X} \circ R_{X,Y}$ to $R_{X,Y}$.
As a result, $\rep(G)$ is an $E_4$ algebra.
\end{rmk}

\begin{thm} \label{thm Muger}
The sylleptic center of $\zv$ is equivalent to $\ve$ as 2-categories.
\end{thm}
\begin{proof}
Let $(\ta, v_{\ta,-})$ be an indecomposable object of the sylleptic center of $\zv$, where $v_{\ta, \ti{X}}: id_{\ti{A} \ti{X}} \Rightarrow R_{\ti{X}, \ti{A}} \circ R_{\ti{A}, \ti{X}}$ gives an isomorphism between the identity and the double braiding.

Take $\ti{X}=T=A(h, H,\psi)$ for $h=1, H=1, \psi=1$.
The half braiding $R_{\ti{A}, T}=\Sigma_{\ti{A}, T}$ from (\ref{eq swap}) since $T$ is concentrated in grading $1$.
So the other half braiding $R_{T, \ti{A}}=\Sigma_{T,\ti{A}}$.
It follows from (\ref{eq braid}) that $\ti{A}$ is concentrated in the grading $1$ since $T$ is the regular representation in $\rep(G)$.
So $\ta$ is an object of $\zv_1 \simeq \rep(G)$.

For any $\ti{X}$, the half braiding $R_{\ti{X}, \ti{A}}=\Sigma_{\ti{X},\ti{A}}$ which implies that $R_{\ti{A}, \ti{X}}=\Sigma_{\ti{A}, \ti{X}}$.
Then $\ti{A}$ has to be the trivial representation in $\rep(G)$ by taking $\ti{X}=T_h=A(h, H,\psi)$ for $h \in G, H=1, \psi=1$.
We have $\ta=A(1, G,\psi)$, where $\psi \in Z^2(G, \Ks)$ is determined by $R_{(\ti{A}|\ti{X},\ti{Y})}$.
Taking $\ti{X}=T_h, \ti{Y}=T_{h'}$, the axiom in (\ref{def Muger}) gives
$$R_{(\ti{A}|\ti{X},\ti{Y})} \cdot v_{\ta, \ti{X}\ti{Y}}=v_{\ta, \ti{X}} \cdot v_{\ta, \ti{Y}},$$
since $R_{(\ti{X},\ti{Y}|\ti{A})}$ is the identity.
This implies that $\psi=d \gamma$, where $\gamma \in Z^1(G, \Ks)$ is a 1-cochain determined by $v_{\ta, T_h}$.
Therefore, the underlying object $\ta$ of $(\ta, v_{\ta,-})$ is isomorphic to the unit $I$ in $\zv$.

We next show that $I_1=(I, v_{I,-})$ and $I_2=(I, v'_{I,-})$ are isomorphic to each other.
Define a 1-morphism $(f, R_{f,-}): I_1 \ra I_2$, where $f=id_I$ and $R_{f,\ti{X}}=v_{I,\ti{X}} \cdot {v'_{I,\ti{X}}}^{-1}$.
It is easy to check that $(f, R_{f,-})$ is well-defined and gives an isomorphism.
So up to isomorphism there is only one indecomposable object $I_0=(I, v_{I,-}), v_{I,\ti{X}}=id_{id_{\ti{X}}}$ in the Sylleptic center.

We finally compute $\End(I_0)$.
Let $f: I \ra I$ be a 1-morphism in $\zv$, i.e. $f \in \End(I) \simeq \op{Rep}(G)$.
It follows from (\ref{def Muger 1-mor}) that $R_{f,\ti{X}}$ is the identity for any $\ti{X}$ since $R_{\ti{X},f}$ is the identity.
So $f$ has to be the trivial representation in $\op{Rep}(G)$ by taking $\ti{X}=T_h$ as above.
We conclude that $\End(I_0) \simeq \vc$.
\end{proof}

Theorem\,\ref{thm Muger} is consistent with the expectation that $\zv$ should be an example of the yet-to-be-defined notion of a {\em unitary modular tensor 2-category}.
Similar to Definition\,\ref{def:m-center}, the notion of the relative sylleptic center of a full subcategory of a braided monoidal 2-category can be defined.
A combination of the proofs of Proposition \ref{prop Muger repG} and Theorem \ref{thm Muger} shows that the sylleptic centralizer of $\rep(G)$ in $\zv$ is equivalent to $\rep(G)$. A unitary modular tensor 2-category $\mathcal{C}$, equipped with a braided monoidal fully faithful embedding $\rep(G)\hookrightarrow \mathcal{C}$, is called a {\it modular extension} of $\rep(G)$. Such a modular extension is called {\it minimal} if the relative sylleptic center of $\rep(G)$ in $\mathcal{C}$ is braided monoidally equivalent to $\rep(G)$. By  Corollary \ref{cor mod}, Proposition \ref{prop Muger repG} and Theorem \ref{thm Muger}, $\zv$ is precisely a minimal modular extension of $\rep(G)$ for $\omega \in Z^4(G, \Ks)$. Motivated by the classification theory of 2+1D symmetry protect topological orders \cite{lkw2} and its 3+1D analogue \cite{cglw,LKW}, we propose the following conjecture.

\begin{conj}
The equivalence classes of minimal modular extensions of $\rep(G)$ are classified by $H^4(G, \Ks)$.
\end{conj}

\end{document}